\documentclass[a4paper,11pt,reqno]{article}

\usepackage[utf8]{inputenc}
\usepackage[T1]{fontenc}
\usepackage{lmodern}
\usepackage[english]{babel}
\usepackage{microtype}

\usepackage{amsmath,amssymb,amsfonts,amsthm}
\usepackage{mathtools,accents}
\usepackage{mathrsfs}
\usepackage{aliascnt}
\usepackage{braket}
\usepackage{bm}

\usepackage[a4paper,margin=3cm]{geometry}
\usepackage[citecolor=blue,colorlinks]{hyperref}

\usepackage{enumerate}
\usepackage{xcolor}
\makeatletter
\g@addto@macro\@floatboxreset\centering
\makeatother

\makeatletter
\def\newaliasedtheorem#1[#2]#3{
  \newaliascnt{#1@alt}{#2}
  \newtheorem{#1}[#1@alt]{#3}
  \expandafter\newcommand\csname #1@altname\endcsname{#3}
}
\makeatother

\numberwithin{equation}{section}

\newtheoremstyle{slanted}{\topsep}{\topsep}{\slshape}{}{\bfseries}{.}{.5em}{}

\theoremstyle{plain}
\newtheorem{theorem}{Theorem}[section]
\newaliasedtheorem{proposition}[theorem]{Proposition}
\newaliasedtheorem{lemma}[theorem]{Lemma}
\newaliasedtheorem{corollary}[theorem]{Corollary}
\newaliasedtheorem{counterexample}[theorem]{Counterexample}

\theoremstyle{definition}
\newaliasedtheorem{definition}[theorem]{Definition}
\newaliasedtheorem{question}[theorem]{Question}
\newaliasedtheorem{conjecture}[theorem]{Conjecture}

\theoremstyle{remark}
\newaliasedtheorem{remark}[theorem]{Remark}
\newaliasedtheorem{example}[theorem]{Example}

\newcommand{\setR}{\mathbb{R}}

\newcommand{\eps}{\varepsilon}

\let\altphi\phi
\let\phi\varphi
\let\varphi\altphi
\let\altphi\undefined

\newcommand{\di}{\mathop{}\!\mathrm{d}}

\newcommand{\bs}{{\rm bs}}

\newcommand{\res}{\mathop{\hbox{\vrule height 7pt width .5pt depth 0pt
\vrule height .5pt width 6pt depth 0pt}}\nolimits}

\DeclareMathOperator{\supp}{supp}

\newcommand{\Ch}{{\sf Ch}}

\DeclareMathOperator{\Lip}{Lip}
\DeclareMathOperator{\Lipb}{Lip_b}

\newcommand{\leb}{\mathscr{L}}

\newcommand{\Prob}{\mathscr{P}}

\newcommand{\dist}{\mathsf{d}}

\newcommand{\meas}{\mathfrak{m}}
\newcommand{\dimnew}{{\rm dim}_{\dist,\meas}}

\DeclareMathOperator{\CD}{CD}
\DeclareMathOperator{\RCD}{RCD}
\newfont{\tmpf}{cmsy10 scaled 2500}

\begin{document}
\title{Short-time behavior of the heat kernel and\\
Weyl's law on $\RCD^*(K,N)$ spaces}
\author{Luigi Ambrosio
\thanks{Scuola Normale Superiore, \url{luigi.ambrosio@sns.it}} \and
Shouhei Honda
\thanks{Tohoku University, \url{shonda@m.tohoku.ac.jp}} \and
David Tewodrose
\thanks{Scuola Normale Superiore,  \url{david.tewodrose@sns.it}}} \maketitle

\begin{abstract} In this paper, we prove pointwise convergence of heat kernels for mGH-convergent sequences of $\RCD^{*}(K,N)$-spaces. We obtain as a corollary results on the short-time behavior of the heat kernel in $\RCD^*(K,N)$-spaces. We use then these results 
to initiate the study of 
Weyl's law in the $\RCD$ setting.
\end{abstract}

\tableofcontents

\section{Introduction}
More than a century ago, H. Weyl gave in \cite{Weyl} a nice description of the asymptotic behavior of the eigenvalues of the Laplacian on bounded domains of $\mathbb{R}^n$ for $n=2,\,3$ (his result was later on extended for any integer $n \ge 2$). More precisely, if $\Omega \subset \mathbb{R}^n$ is a bounded domain, it is well-known that the spectrum of (minus) the Dirichlet Laplacian on $\Omega$ is a sequence of positive numbers $(\lambda_{i})_{i \in \mathbb{N}^{*}}$ such that $\lambda_{i} \rightarrow +\infty$ as $i\to +\infty$. Weyl proved that
$$ \lim_{\lambda\to+\infty}\frac{N(\lambda)}{\lambda^{n/2}} = \frac{\omega_{n}}{(2\pi)^n} \mathcal{L}^n(\Omega)$$
where $N(\lambda) =\sharp \{i \in \mathbb{N}:\ \lambda_i\le \lambda\}$ (the eigenvalues being counted with multiplicity), 
$\omega_n$ is the volume of the $n$-dimensional euclidean unit ball, and $\mathcal{L}^n(\Omega)$ is the $n$-dimensional Lebesgue measure of $\Omega$. This result is known as Weyl's law. It has been widely used to tackle some physical problems, and several refinements were found after Weyl's first article. For a complete overview of the history of Weyl's law and its refinements, we refer to \cite{ArendtNittkaPeterSteiner}.

Among the possible generalizations of Weyl's law, one can replace the bounded domain $\Omega \subset \mathbb{R}^n$ by a $n$-dimensional compact closed manifold. The Laplacian is then replaced by the Laplace-Beltrami operator of the manifold, and the term $\mathcal{L}^n(\Omega)$ is replaced by $\mathcal{H}^n(M)$, where $\mathcal{H}^n$ denotes the $n$-dimensional Hausdorff measure. It has been proved by B. Levitan in \cite{Levitan} that Weyl's law is still true in that case.

Another generalization concerns compact Riemannian manifolds $(M,g)$ equipped with the distance $\dist$ induced by the metric $g$ and a measure with positive smooth density $e^{-f}$ with respect to the volume measure $\mathcal{H}^n$. For such spaces
$(M,\dist,e^{-f}\mathcal{H}^n)$, called weighted Riemannian manifolds,  one has
\begin{equation}\label{eq:intro weyl2}
\lim_{\lambda \to +\infty}\frac{N_{(M, \dist, e^{-f}\mathcal{H}^n)}(\lambda)}{\lambda^{n/2}}=\frac{\omega_n}{(2\pi)^n}\mathcal{H}^n(M),
\end{equation}
holds, where $N_{(M, \dist, e^{-f}\mathcal{H}^n)}(\lambda)$ denotes the counting function of the (weighted) Laplacian $\Delta^f:=\Delta -\langle \nabla f, \nabla \cdot \rangle$ of $(M, \dist, e^{-f}\mathcal{H}^n)$. This result is a consequence of \cite{Hormander}.
We stress that in the asymptotic behavior (\ref{eq:intro weyl2}) the information of the weight, $e^{-f}$, disappears (as we obtain by different means in Example~\ref{ex:weighted weyl}). This sounds surprising: the Hausdorff dimension is a purely metric notion, whereas the Laplace-Beltrami operator on weighted Riemannian manifolds and more generally on $\RCD^*(K, N)$-spaces does depend on the reference measure. 

In this paper we focus on infinitesimally Hilbertian metric measure spaces with Ricci curvature bounded from below, the so-called $\RCD$-spaces.  The curvature-dimension condition 
$\CD(K,N)$ was independently formulated in terms of optimal transport by Sturm in \cite{Sturm06} and Lott-Villani in \cite{LottVillani}. The $\CD$ condition extends to a non-smooth setting the Riemannian notion of Ricci curvature bounded below. Indeed, for given $K \in \mathbb{R}$ and $N \in [1,+\infty]$, a Riemannian manifold satisfies the $\CD(K,N)$ condition if and only if it has Ricci curvature bounded below by $K$ and dimension bounded above by $N$. The $\CD$ condition is also stable under Gromov-Hausdorff convergence: any metric measure space obtained as a measured Gromov-Hausdorff limit of a sequence of Riemannian manifolds with Ricci curvature bounded below by $K$ and dimension bounded above by $N$ satisfies the $\CD(K,N)$ condition. Such limit spaces are called Ricci limit spaces in the sequel. 

In more recent times, two main requirements were added to that theory, namely, the $\CD^{*}$ condition introduced in \cite{BacherSturm} and the infinitesimal Hilbertianity introduced in \cite{AmbrosioGigliSavare14}, giving rise to the study of the so-called $\RCD$ (resp. $\RCD^{*}$) spaces which are by definition infinitesimally Hilbertian spaces satisfying the $\CD$ (resp. $\CD^{*}$) condition. All these notions are stable under Gromov-Hausdorff convergence. See also the papers \cite{AmbrosioGigliSavare15}, \cite{ErbarKuwadaSturm}, 
\cite{AmbrosioMondinoSavare}, where the $\RCD(K,\infty)/\RCD^*(K,N)$ theories
have been proved to be essentially equivalent to the Bakry-Emery theory. The latter, based on diffusion operators and 
Bochner's inequality, for weighted Riemannian manifolds $(M, \dist, e^{-f}\mathcal{H}^n)$ reads as follows:
\begin{equation}\label{eq:RCD smooth metric measure space}
\mathrm{Ric}_M+\mathrm{Hess}_f -\frac{\nabla f \otimes \nabla f}{N-n} \ge Kg_M.
\end{equation}

Let $(X, \dist, \meas)$ be a compact $\RCD^* (K, N)$-space. The main result of this paper is a sharp criterion 
(see \eqref{eq:intro weyl cri} below) for the validity of Weyl's law on $(X, \dist, \meas)$. The authors do not know whether 
there exist $\RCD^*(K, N)$-spaces which do not satisfy this criterion, since all known examples satisfy it.

As observed in (\ref{eq:intro weyl2}) it is expected that the asymptotic behavior of the counting function $N_{(X, \dist, \meas)}(\lambda)$ is 
not related to the reference measure $\meas$, but rather to the Hausdorff measure $\mathcal{H}^\ell$, where $\ell$ is the Hausdorff dimension of $(X, \dist)$. 

In order to introduce the precise statement of our criterion let us recall Mondino-Naber's result \cite{MondinoNaber}:
$$
\meas (X \setminus \bigcup_{i=1}^{[N]}\mathcal{R}_i )=0,
$$
where $[N]$ is the integer part of $N$ and $\mathcal{R}_i$ is the $i$-dimensional regular set of $(X, \dist, \meas)$. Recall that the $i$-dimensional regular set of $(X,\dist)$ (Definition~\ref{def:regular}) is the set of points of $X$ admitting a unique tangent cone isometric to $(\mathbb{R}^{i},d_{eucl},c_i\mathcal{H}^i)$. More recently, building on \cite{DePhilippisRindler}, more than one group of authors have shown that 
$\meas$-almost all of $\mathcal{R}_{i}$ can be covered 
by bi-Lipschitz charts defined in subsets of $\mathbb{R}^{i}$, and that the restriction of the reference measure to each $\mathcal{R}_{i}$ is absolutely continuous w.r.t. $\mathcal{H}^i$ (Proposition~\ref{thm:propRk} and \cite{KellMondino}, \cite{DePhilippisMarcheseRindler}, \cite{GigliPasqualetto}).

Let us define the maximal regular dimension $\dimnew(X)$ of $(X,\dist,\meas)$ as the largest integer $k$
such that $\meas(\mathcal{R}_k)>0$. We are now in a position to introduce our main result. 
We prove first in Theorem~\ref{thm:RN} that, setting $k=\dimnew(X)$, the limit
$$
\lim_{r \to 0^+}\frac{r^k}{\meas (B_r(x))}
$$
exists and is finite for $\meas$-a.e. $x \in X$. Then our criterion (Theorem~\ref{thm:weyl}) can be stated as follows:
\begin{equation}\label{eq:intro weyl cri}
\lim_{r \to 0^+}\int_X\frac{r^k}{\meas (B_r(x))} \dist \meas =\int_X\lim_{r \to 0^+}\frac{r^k}{\meas (B_r(x))} \dist \meas <\infty
\end{equation}
if and only if
\begin{equation}\label{eq:intro weyl law}
\lim_{\lambda \to +\infty}\frac{N_{(X, \dist, \meas)}(\lambda)}{\lambda^{k/2}}=\frac{\omega_k}{(2\pi)^k}\mathcal{H}^k({\cal R}_k^*)<\infty,
\end{equation}
where $\mathcal{R}_k^*\subset\mathcal{R}_k$ denotes a suitable reduced regular set (defined in Theorem~\ref{thm:RN})
such that $\meas(\mathcal{R}_k\setminus\mathcal{R}_k^*)=0$ and $\meas\res\mathcal{R}_k^*$ and $\mathcal{H}^k\res
\mathcal{R}_k^*$ are mutually absolutely continuous (in particular $\mathcal{H}^k(\mathcal{R}_k^*)>0$).
%{\color{red} The authors do not know whether some $\RCD^*(K, N)$-space satisfy $\mathcal{R}_k \neq \mathcal{R}_k^*$.}

Note that with this criterion, the asymptotic behavior of $N_{(X, \dist, \meas)}(\lambda)$ (including the growth order) is determined by the sole top-dimensional reduced regular set.
As typical examples, thanks to the dominated convergence theorem, the criterion is automatically satisfied when $k=N$, or when the metric measure structure is Ahlfors regular. As a consequence, we  obtain a new result, namely Weyl's law for finite dimensional compact Alexandrov spaces (Corollaries \ref{cor:weyls law} and \ref{cor:weyl alex}). We can also provide examples (see \eqref{eq:intro exa} below) such that 
$k<N$ and Ahlfors regularity fails.

On the other hand, it is worth pointing out that from the viewpoint of $\RCD$-theory the least number $N_{{\rm min}}$ such that 
$(X,\dist,\meas)\in\RCD^*(K,N_{{\rm min}})$ for some $K\in\mathbb{R}$ 
might be naturally regarded as another dimension of $(X, \dist, \meas)$ (indeed, $N_{{\rm min}}=n$ for weighted Riemannian
manifolds $(M, \dist, e^{-f}\mathcal{H}^n)$).
However, in general $N_{{\rm min}}$ is not equal to the Hausdorff dimension of $(X, \dist )$ and need not be 
related to the asymptotic behaviour of $N_{(X, \dist, \meas)}(\lambda)$, as the following example shows:
for $N \in (1,+\infty)$, let us consider the metric measure space
\begin{equation}\label{eq:intro exa}
(X, \dist, \meas):=([0, \pi], \dist_{[0, \pi]},\sin^{N-1}t\dist t).
\end{equation}
It is known that $(X, \dist, \meas)$ is a $\RCD^*(N-1, N)$-space (see for instance \cite{CavallettiMondino15a}). 
Moreover, since $\meas (B_r(x)) \sim r$ for $x \in (0, \pi)$ and $\meas (B_r(x)) \sim r^N$ for $x \in \{0, \pi\}$ as $r \to 0^+$, for
this metric measure structure one has
$N_{{\rm min}}=N$, because Bishop-Gromov inequality for $\RCD^*(\hat{K}, \hat{N})$-spaces implies 
a positive lower bound on $\meas (B_r(x))/r^{\hat{N}}>0$.  

It turns out that our criterion can be applied to $(X, \dist, \meas)$, as \eqref{eq:intro weyl cri} 
holds by the dominated convergence theorem. Thus
\begin{equation}\label{eq:exa}
\lim_{\lambda \to +\infty}\frac{N_{(X, \dist, \meas)}(\lambda)}{\lambda^{1/2}}=\frac{\omega_1}{2\pi}\mathcal{H}^1((0, \pi))=1
\end{equation}
and the number $N$ does not appear in (\ref{eq:exa}).
Note that (\ref{eq:exa}) is also new and that the same asymptotic behavior in (\ref{eq:exa}) for $N=1$ (i.e. the metric measure space is $([0, \pi], \dist_{[0, \pi]}, \mathcal{H}^1))$ is well-known as Weyl's law on $[0, \pi]$ associated with homogeneous Neumann boundary conditions. See Example~\ref{rem:example1} for more details.\\

Let us conclude this introduction by pointing out that our technique is based on a pointwise convergence of heat kernels for Gromov-Hausdorff converging sequences of $\RCD^{*}(K,N)$ spaces (Theorem~\ref{thm:heat kernel}) which is a generalization of Ding's Riemannian results \cite{Ding}. A main advantage of our approach is the use of regularity theory of heat flows that avoids some technical
difficulties  of spectral theory. As a corollary, we obtain a precise short-time diagonal estimate of the heat kernel $p(x,x,t)$ on the regular sets of a compact $\RCD^{*}(K,N)$ space $(X,\dist,\meas)$. In fact, \eqref{eq:intro weyl cri} allows to turn this estimate into a trace formula (see the proof of Theorem~\ref{thm:weyl})
$$ \lim_{t\to 0^+} (4\pi t)^{k/2}\int_{X} p(x,x,t) \dist \meas(x) = \mathcal{H}^k (\mathcal{R}_{k}^{*}),$$ where $k=\dim_{\dist, \meas}(X)$, leading naturally to Weyl's law.\\

The paper is organized as follows. In Section 2 we recall the notions on $\RCD^{*}(K,N)$ spaces that shall be used in the sequel and give some useful lemmas. Section 3 begins with the treatment of weak/strong $L^2$-convergence for sequences of functions defined on converging 
$\RCD^{*}(K,N)$ spaces. Then, using the Gaussian estimates \eqref{eq:gaussian} of the heat kernel on $\RCD^{*}(K,N)$-spaces established in \cite{JiangLiZhang}, we prove the pointwise convergence of heat kernels defined on a converging sequence of $\RCD^{*}(K,N)$ spaces, and we deduce from this fact the short-time estimate of the heat kernel on regular sets of $\RCD^{*}(K,N)$ spaces. Section 4 is devoted to the study of Weyl's law, first in a weak form (Theorem~\ref{thm:weakweyl}) and then strengthening the assumptions in a stronger and more classical form (Theorem~\ref{thm:weyl}). The rest of this section is dedicated to examples and applications (especially to compact Alexandrov spaces). 

After completing our paper we learned of an independent work \cite{ZhangZhu2} by H-C. Zhang and X-P. Zhu on Weyl's law
in the setting of $\RCD^*(K,N)$ spaces. The paper is based on a local analysis, along the same lines of \cite{Ding}, and provides 
sufficient conditions for the validity of Weyl's law, different from our sharp criterion of Theorem~\ref{thm:weyl}. 

\smallskip
\noindent
{\bf Acknowledgement.} The first and third author acknowledge the support of the PRIN2015 MIUR Project ``Calcolo delle Variazioni''. 
The second author acknowledges the support of the JSPS Program for Advancing 
Strategic International Networks to Accelerate the Circulation of Talented Researchers, the
Grant-in-Aid for Young Scientists (B) 16K17585 and the Scuola Normale Superiore for warm hospitality.

\section{Notation and preliminaries about $\RCD^*(K,N)$ spaces}

Let us recall basic facts about Sobolev spaces and heat flow in metric measure spaces 
$(X,\dist,\meas)$, see \cite{AmbrosioGigliSavare13} and \cite{Gigli1} for a more systematic treatment of this topic. The
so-called Cheeger energy
$\Ch:L^2(X,\meas)\to [0,+\infty]$ is the convex and $L^2(X,\meas)$-lower semicontinuous functional defined as follows:
\begin{equation}\label{eq:defchp}
\Ch(f):=\inf\left\{\liminf_{n\to\infty}\frac 12\int_X{\rm Lip}_a^2(f_n)\di\meas:\ \text{$f_n\in\Lipb(X)\cap L^2(X,\meas)$, $\|f_n-f\|_2\to 0$}\right\}. 
\end{equation}
The original definition in \cite{Cheeger} involves generalized upper gradients of $f_n$ in place of their asymptotic Lipschitz constant
$${\rm Lip}_a(f):=\lim_{r\to 0^+}{\rm Lip}(f,B_r(x)),$$ 
but many other pseudo gradients (upper gradients, or the slope ${\rm lip}(f)\leq{\rm Lip}_a(f)$, which is a particular upper gradient) can be used and all of them lead to the same definition, see \cite{AmbrosioColomboDiMarino} and the discussion in \cite[Remark~5.12]{AmbrosioGigliSavare13}). 

The Sobolev space $H^{1,2}(X,\dist,\meas)$ is simply defined as the finiteness domain of $\Ch$. When endowed with the norm
$$
\|f\|_{H^{1,2}}:=\left(\|f\|_{L^2(X,\meas)}^2+{\color{blue}2}\Ch(f)\right)^{1/2}
$$
this space is Banach, and reflexive if $(X,\dist)$ is doubling (see \cite{AmbrosioColomboDiMarino}).  
The Sobolev space is Hilbert if $\Ch$ is a quadratic form. We say that a metric measure space
$(X,\dist,\meas)$ is infinitesimally Hilbertian if $\Ch$ is a quadratic form.
  
By looking at minimal relaxed slopes and by a polarization procedure, one can then define a {\it Carr\'e du champ}
$$
\Gamma:H^{1,2}(X,\dist,\meas)\times H^{1,2}(X,\dist,\meas)\rightarrow L^1(X,\meas)
$$
playing in this abstract theory the role of the scalar product between gradients. In infinitesimally Hilbertian metric measure
spaces the $\Gamma$ operator
satisfies all natural symmetry, bilinearity, locality and chain rule properties, and provides integral representation to
$\Ch$: $2\Ch(f)=\int_X \Gamma(f,f)\,\dist\meas$ for all $f\in H^{1,2}(X,\dist,\meas)$. We can also define a densely
defined operator $\Delta:D(\Delta)\to L^2(X,\meas)$ by
$$
f\in D(\Delta)\,\,\,\Longleftrightarrow\,\,\,
\exists h:=\Delta f\in L^2(X,\meas)\,\,\text{s.t. } \int_X hg\dist\meas=-\int_X \Gamma(f,g)\dist\meas\,\,\,\forall g\in H^{1,2}(X,\dist,\meas).
$$

Another object canonically associated to $\Ch$ and then to the metric measure structure is the heat flow $h_t$, defined as the
$L^2(X,\meas)$ gradient flow of $\Ch$; even in general metric measure structures one can use the Brezis-Komura theory of 
gradient flows of lower semicontinuous functionals in Hilbert spaces to provide existence and uniqueness of this gradient flow. 
In the special case of infinitesimally Hilbertian 
metric measure spaces, this provides a linear, continuous and self-adjoint contraction
semigroup $h_t$ in $L^2(X,\meas)$ with the Markov property, characterized by:
$t\mapsto h_t f$ is locally absolutely continuous in $(0,+\infty)$ with values in $L^2(X,\meas)$ and
$$
\frac{d}{dt}h_t f=\Delta h_t f\quad\text{for $\leb^1$-a.e. $t\in (0,+\infty)$}
$$
for all $f\in L^2(X,\meas)$. Thanks to the Markov property, this semigroup has a unique $L^p$ continuous extension
from $L^2\cap L^p$ to $L^p$, $1\leq p<+\infty$, and by duality one defines also the $w^*$-continuous 
extension to $L^\infty(X,\meas)$.

In order to introduce the class of $\RCD(K,\infty)$ and $\RCD^*(K,N)$ metric measure spaces we follow the
$\Gamma$-calculus point of view, based on Bochner's inequality, because this is the point of view more relevant
in the proof of heat kernel estimates, Li-Yau inequalities, etc. The equivalence with the Lagrangian point
of view, based on the theory of optimal transport is discussed in \cite{AmbrosioGigliSavare15} (in the case $N=\infty$) and in
\cite{ErbarKuwadaSturm}, \cite{AmbrosioMondinoSavare} (in the case $N<\infty$). The latter point of view does not play a 
role in this paper, but it plays indeed a key role in the proof of the results we need, mainly taken from \cite{GigliMondinoSavare13} 
and \cite{MondinoNaber}. 

\begin{definition} [$\RCD$ spaces]\label{def:RCDspaces} Let $(X,\dist,\meas)$ be a metric measure space, with $(X,\dist)$ complete, satisfying
\begin{equation}\label{eq:Grygorian}
\meas\bigl(B_r(\bar x)\bigr)\leq c_1 e^{c_2r^2}\qquad\forall r>0
\end{equation}
for some $c_1,\,c_2>0$ and $\bar x\in X$ and the so-called Sobolev to Lipschitz property:  any  
$f\in H^{1,2}(X,\dist,\meas)\cal L^\infty(X,\meas)$ with $\Gamma(f)\leq 1$ $\meas$-a.e. in $X$ 
has a representative in $\tilde{f}\in\Lipb(X)$, with $\Lip(\tilde f)\leq 1$.

For $K\in\setR$, 
we say that $(X,\dist,\meas)$ is a $\RCD(K,\infty)$ metric measure space if, 
for all $f\in H^{1,2}(X,\dist,\meas)\cap D(\Delta)$ with $\Delta f\in H^{1,2}(X,\dist,\meas)$,
 Bochner's inequality
$$
\frac 12\Delta\Gamma(f)\geq \Gamma(f,\Delta f)+K\Gamma(f) 
$$
holds in the weak form
$$
\frac 12\int \Gamma(f)\Delta\phi\dist\meas\geq
\int\phi(\Gamma(f,\Delta f)+K\Gamma(f))\dist\meas 
\quad\forall \phi\in D(\Delta)\,\,\text{with $\phi\geq 0$,}\,\,\Delta\phi\in L^\infty(X,\meas).
$$

Analogously, for $K\in\setR$ and $N>0$, we say that $(X,\dist,\meas)$ is a $\RCD^*(K,N)$ metric 
measure space if, for all $f\in H^{1,2}(X,\dist,\meas)\cap D(\Delta)$ with $\Delta f\in H^{1,2}(X,\dist,\meas)$, 
Bochner's inequality
$$
\frac 12\Delta\Gamma(f)\geq \Gamma(f,\Delta f)+\frac 1N(\Delta f)^2+K\Gamma(f) 
$$
holds in the weak form
$$
\frac 12\int \Gamma(f)\Delta\phi\dist\meas\geq
\int\phi(\Gamma(f,\Delta f)+\frac 1N(\Delta f)^2+K\Gamma(f))\dist\meas 
$$
for all $\phi\in D(\Delta)$ with $\phi\geq 0$ and $\Delta\phi\in L^\infty(X,\meas)$.
\end{definition}

The assumption \eqref{eq:Grygorian} is needed to ensure stochastic completeness, namely the property $h_t1=1$.
For our purposes it will be convenient not to add the assumption that $X=\supp\meas$, made in some other papers
on this subject.
Neverthless, it is obvious that $(X,\dist,\meas)$ is $\RCD(K,\infty)$ (resp. $\RCD^*(K,N)$) if and only if
$(X,\dist,\supp\meas)$ is $\RCD(K,\infty)$ (resp. $\RCD^*(K,N)$).

For $\RCD(K,\infty)$ spaces it is proved in \cite{AmbrosioGigliSavare14} that the dual semigroup $\tilde h_t$, acting on 
the space $\Prob_2(X)$ of probability measures with finite quadratic moments, is $K$-contractive and maps for all $t>0$ 
$\Prob_2(X)$ into measures absolutely continuous w.r.t. $\meas$, with finite logarithmic entropy. Setting then 
$$
\tilde h_t\delta_x=p(x,\cdot,t)\meas\qquad x\in X,\,\,t>0
$$
this provides a version of the heat kernel $p(x,y,t)$ in this class of spaces
(defined for any $x$ in $\supp\meas$, up to a $\meas$-negligible set of points $y$), so that
$$
h_tf(x)=\int_X p(x,y,t)f(y)\dist\meas\qquad\forall f\in L^2(X,\meas).
$$
In $\RCD^*(K,N)$ spaces with $N<\infty$,
thanks to additional properties satisfied by the metric measure structure,
one can find a version of $p$ continuous in $\supp\meas\times\supp\meas\times (0,+\infty)$, 
as illustrated in the next section.

\begin{definition}[Rectifiable sets]
Let $(X,d)$ be a metric space and let $k\geq 1$ be an integer. We say that $S\subset X$ is
countably $k$-rectifiable if there exist at most countably many sets $B_i\subset\setR^k$ 
and Lipschitz maps $f_i:B_i\to X$ such that $S\subset\cup_i f_i(B_i)$.\\
For a nonnegative Borel measure $\mu$ in $X$ (not necessarily $\sigma$-finite), we say
that $S$ is $(\mu,k)$-rectifiable if there exists a countably $k$-rectifiable set $S'\subset S$
such that $\mu^*(S\setminus S')=0$, namely $S\setminus S'$ is contained in a
$\mu$-negligible Borel set.
\end{definition}

In the next proposition we recall some basic differentiation properties of measures.

\begin{proposition} If $\mu$ is a locally finite and nonnegative Borel measure in $X$ and $S\subset X$ is a Borel set, one has
\begin{equation}\label{eq:density1}
\mu(S)=0\quad\Longrightarrow\quad\mu(B_r(x))=o(r^k)\,\,\,\text{for $\mathcal{H}^k$-a.e. $x\in S$.}
\end{equation}
In addition, 
\begin{equation}\label{eq:density3}
\mu(S)=0,\,\,S\subset\{x:\ \limsup_{r\to 0^+}\frac{\mu(B_r(x))}{r^k}>0\}\quad\Longrightarrow\quad
\mathcal{H}^k(S)=0.
\end{equation}
Finally, if $\mu=f\mathcal{H}^k\res S$ with $S$ countably $k$-rectifiable, one has
\begin{equation}\label{eq:densityrett}
\lim_{r\to 0^+}\frac{\mu(B_r(x))}{\omega_kr^k}=f(x)\qquad\text{for $\mathcal{H}^k$-a.e. $x\in S$.}
\end{equation}
\end{proposition}
\begin{proof} The proof of \eqref{eq:density1} and \eqref{eq:density3} can be found for instance
in \cite[2.10.19]{Federer} in a much more general context. See also \cite[Theorem~2.4.3]{AmbrosioTilli} for more specific statements
and proofs. The proof of \eqref{eq:densityrett} is given in \cite{Kirchheim} when $\mu=\mathcal{H}^k\res S$, with
$S$ countably $k$-rectifiable and having locally finite $\mathcal{H}^k$-measure (the proof uses the fact that 
for any $\eps>0$ we can cover $\mathcal{H}^k$-almost all of $S$ by sets $S_i$ which are biLipschitz
deformations, with biLipschitz constants smaller than $1+\eps$, of $(\setR^i,\|\cdot\|_i)$, for suitable
norms $\|\cdot \|_i$). In the general case a simple comparison argument gives the result.
\end{proof}

We conclude this section with two auxiliary results.

\begin{lemma}\label{lem:dct}
Let $f_i, \,g_i,\, f,\, g \in L^1(X, \meas)$.
Assume that $f_i, \,g_i \to f,\, g$ $\meas$-a.e. respectively, that $|f_i|\le g_i$ $\meas$-a.e., 
and that $\lim_{i \to \infty}\|g_i\|_{L^1}=\|g\|_{L^1}$.
Then $f_i \to f$ in $L^1(X, \meas)$. 
\end{lemma}
\begin{proof} Obviously $|f|\leq g$ $\meas$-a.e. 
Applying Fatou's lemma for $h_i:= g_i+g-|f_i-f| \ge 0$ yields
$$
\int_X\liminf_{i \to \infty}h_i\dist \meas \le \liminf_{i \to \infty}\int_Xh_i\dist \meas.
$$
Then by assumption the left hand side is equal to $2\|g\|_{L^1}$, 
and the right hand side is equal to $2\|g\|_{L^1}-\limsup_i\|f_i-f\|_{L^1}$. It follows that 
$\limsup_i\|f_i-f\|_{L^1}=0$, which completes the proof.
\end{proof}

The proof of the next classical result can be found, for instance, in \cite[Sec.~XIII.5, Theorem~2]{Feller}.

\begin{theorem}[Karamata's Tauberian theorem]\label{thm:karamata}
Let $\nu$ be a nonnegative and locally finite measure in $[0,+\infty)$ and set
$$
\hat\nu(t):=\int_{[0,+\infty)} e^{-\lambda t}\dist\nu(\lambda)\qquad t>0.
$$
Then, for all $\gamma>0$ and $a\in [0,+\infty)$ one has
$$
\lim_{t\to 0^+}t^\gamma\hat\nu(t)=a\qquad\Longleftrightarrow
\qquad\lim_{\lambda\to +\infty}\frac{\nu([0,\lambda])}{\lambda^\gamma}=
\frac{a}{\Gamma(\gamma+1)}.
$$
In particular, if $\gamma=k/2$ with $k$ integer, the limit in the right hand side can be written as
$a\omega_k/\pi^{k/2}$. 
\end{theorem}

\begin{remark}[One-sided versions]\label{rem:tauberian}
More generally we shall prove in the last section of the paper the so-called Abelian one-sided implications and inequalities:
\begin{equation}\label{eq:abe1}
\liminf_{t\to 0^+}t^\gamma\hat\nu(t)\geq \Gamma(\gamma+1)\liminf_{\lambda \to +\infty}\frac{\nu([0,\lambda])}{\lambda^\gamma},
\end{equation}
\begin{equation}\label{eq:abe2}
\limsup_{\lambda\to +\infty}\frac{\nu([0,\lambda])}{\lambda^\gamma}<+\infty\quad\Longrightarrow\quad
\limsup_{t\to 0^+}t^\gamma\hat\nu(t)\leq\Gamma(\gamma+1)
\limsup_{\lambda \to +\infty}\frac{\nu([0,\lambda])}{\lambda^\gamma}
\end{equation}
as well as the so-called Tauberian one-sided implications and inequalities
\begin{equation}\label{eq:tau1}
\limsup_{\lambda\to +\infty}\frac{\nu([0,\lambda])}{\lambda^\gamma}\leq 
e\limsup_{t\to 0^+}t^\gamma\hat\nu(t),
\end{equation}
\begin{equation}\label{eq:tau2}
\liminf_{t\to 0^+}t^\gamma\hat\nu(t)>0,\,\,\limsup_{t\to 0^+}t^\gamma\hat\nu(t)<+\infty\qquad\Longrightarrow\qquad
\liminf_{\lambda \to +\infty}\frac{\nu([0,\lambda])}{\lambda^\gamma}>0.
\end{equation}
\end{remark}

Notice that \eqref{eq:tau2} is not quantitative, and requires both bounds on the $\liminf$
and the $\limsup$, see Remark~\ref{rem:moments} for an additional discussion.

\section{Pointwise convergence of heat kernels}

From now on, $K\in\mathbb{R}$ and $N \in [1,+\infty)$. Let us fix a pointed measured Gromov-Hausdorff (mGH for short in the sequel) convergent sequence $(X_i, \dist_i, x_i, \meas_i) \stackrel{mGH}{\to} (X, \dist, x, \meas)$ of $\RCD^*(K, N)$-spaces. This means that there exist sequences of positive numbers $\epsilon_i \to 0$, $R_i \uparrow \infty$, and of Borel maps $\phi_i:B_{R_i}(x_i) \to X$ such that:
\begin{itemize}
\item[(a)] $|\dist_i(x, y)-\dist (\phi_i(x), \phi_i(y))|<\epsilon_i$ for any $i$ and all $x,\, y \in B_{R_i}(x_i)$, so that
$ B_{R_i -\epsilon_i}(\phi_i(x_i)) \subset B_{\epsilon_i}(\phi_i(B_{R_i}(x_i)))$;
\item[(b)] $\phi_i(x_i) \to x$ in $X$ as $i \to \infty$ (we denote it by $x_i \stackrel{GH}{\to} x$ for short);
\item[(c)] $(\phi_i)_{\sharp}\meas_i \stackrel{C_\bs(X)}{\rightharpoonup} \meas$.
\end{itemize}

In statement (c) we have denoted by $C_\bs(X)$ the space of continuous functions with bounded support, and by $ f_\sharp $
the push forward operator between measures induced by a Borel map $f$. We shall use this notation also in the sequel and we call
weak convergence the convergence in duality with $C_\bs(X)$.

Since $\meas_i$ are uniformly doubling (it follows directly from Bishop-Gromov inequality, known to be true even in the $\CD^*(K,N)$ case), the mGH-convergence is equivalent to the pointed measured Gromov (pmG for short) convergence introduced in \cite{GigliMondinoSavare13}. Recall that ``$(X_{i}, \dist_i, x_i, \meas_i)$ pmG-converges to $(X, \dist, x, \meas)$'' means that there exist a doubling and complete metric space 
$\mathbb{X}$ and isometric embeddings $\psi_i:X_i \hookrightarrow \mathbb{X}$, $\psi: X \hookrightarrow \mathbb{X}$ such that $\psi_i(x_i) \to \psi (x)$ in $\mathbb{X}$ as $i \to \infty$ (we also write $x_i \stackrel{GH}{\to} x$ for short) and such that $(\psi_i)_{\sharp}\meas_i \stackrel{C_\bs(\mathbb{X})}{\rightharpoonup}  (\psi)_{\sharp}\meas$.
See \cite[Theorem 3.15]{GigliMondinoSavare13} for the proof of the equivalence.

Since all objects we are dealing with are invariant under isometric and measure-preserving embeddings, we identify in the sequel
$(X_i, \dist_i, x_i, \meas_i)$ with its image by $\psi_i$, i.e. $(X_i, \dist_i, x_i, \meas_i)=(\psi_i(X_i), \dist, \psi_i(x_i), (\psi_i)_{\sharp}\meas_i)$.
So, in the sequel the complete and doubling space $(X,\dist)$ will be fixed (playing the role of $\mathbb{X}$), 
and we denote by  $X_i\subset X$ the supports of the measures $\meas_i$, weakly convergent in $X$ to $\meas$. Because of this, we also
use the simpler notation $y_i\to y$ for $y_i \stackrel{GH}{\to} y$. We recall that complete and doubling spaces are proper (i.e.
bounded closed sets are compact), hence separable.

Under this notation let us recall the definition of $L^2$-strong/weak convergence of functions with respect to the mGH-convergence.
The following formulation is due to \cite{GigliMondinoSavare13} and \cite{AmbrosioStraTrevisan}, which fits the pmG-convergence well.
Other good formulations of $L^2$-convergence, in connection with mGH-convergence, can be found in \cite{Honda2, KuwaeShioya03}.
However in our setting these formulations are equivalent by the volume doubling condition (e.g. \cite[Proposition 3.3]{Honda7}). 

\begin{definition}[$L^2$-convergence of functions with respect to variable measures]\label{def:l2}
\item[]
\begin{enumerate}
\item{($L^2$-strong/weak convergence)} We say that $f_i \in L^2(X_i, \meas_i)$ \textit{$L^2$-weakly converge to $f \in L^2(X, \meas )$} 
if $\sup_i\|f_i\|_{L^2}<\infty$ and $f_i\meas_i \stackrel{C_\bs(X)}{\rightharpoonup} f\meas$.
Moreover, we say that $f_i\in L^2(X_i,\meas_i)$ \textit{$L^2$-strongly converge to $f\in L^2(X,\meas)$} if 
$f_i$ $L^2$-weakly converge to $f$ with $\limsup_{i \to \infty}\|f_i\|_{L^2}\le \|f\|_{L^2}$. 
\item{($L^2_{\mathrm{loc}}$-strong/weak convergence)}
%We say that $f_i \in L^2_{\mathrm{loc}}(X_i, \meas_i)$ \textit{$L^2_{\mathrm{loc}}$-weakly (or strongly, respectively) converge to $f \in L^2_{\mathrm{loc}}(X, \meas )$} if $1_{B_R(x_i)}f_i$ $L^2_{\mathrm{loc}}$-weakly (or strongly, respectively) converge to $1_{B_R(x)}f$ for any $R>0$.
We say that $f_i \in L^2_{\mathrm{loc}}(X_i, \meas_i)$ \textit{$L^2_{\mathrm{loc}}$-weakly (or strongly, respectively) converge to $f \in L^2_{\mathrm{loc}}(X, \meas )$} if $\zeta f_i$ $L^2_{\mathrm{loc}}$-weakly (or strongly, respectively) converge to $\zeta f$ for any 
$\zeta\in C_\bs(X)$.
\end{enumerate}
\end{definition}

\begin{proposition}\label{prop:equivalence}
Let $f_i \in C^0(X_i)$ and $f \in C^0(X)$, with $X$ proper and $$\sup_i\sup_{X_i\cap B_R(x_i)}|f_i|<+\infty
\qquad\forall R>0.$$ Assume that $\{f_i\}_i$ is locally equi-continuous, i.e. for any $\epsilon>0$ and any $R>0$ there exists $\delta>0$ 
independent of $i$ such that 
\begin{equation}\label{eq:equi cont}
(y, z) \in (X_i\cap B_R(x_i))^2\,\,\,\dist(y,z)<\delta\quad\Longrightarrow\quad |f_i(y)-f_i(z)|<\epsilon.
\end{equation}
Then the following are equivalent:
\begin{enumerate}
\item[(1)] $\lim\limits_{k\to \infty}f_{i(k)}(y_{i(k)}) = f(y)$ whenever $y\in\supp\meas$, $i(k)\to\infty$ and $y_{i(k)} \in X_{i(k)}\to y$,
\item[(2)] $f_i$ $L^2_{\mathrm{loc}}$-weakly converge to $f$,
\item[(3)] $f_i$ $L^2_{\mathrm{loc}}$-strongly converge to $f$.
\end{enumerate}
\end{proposition}
\begin{proof}
We prove the implication from (1) to (3) and from (2) to (1), since the implication from (3) to (2) is trivial.

Assume that (2) holds, let $\epsilon>0$ and let $y_i \to y$. Take $\zeta$ nonnegative, with support contained
in $B_\delta(y)$ and with $\int\zeta\dist\meas=1$. Thanks to \eqref{eq:equi cont} and the continuity of $f$, for $\delta$ 
sufficiently small we have
$$
(f_i(y_i)-\epsilon)\int\zeta\dist\meas_i\leq \int \zeta f_i\dist\meas_i\leq (f_i(y_i)+\eps)\int\zeta\dist\meas_i
\qquad
f(y)-\epsilon\leq\int \zeta f\dist\meas\leq f(y)+\eps
$$
Since $\int\zeta f_i\dist\meas_i\to\int \zeta f\dist\meas$ and $\int\zeta\dist\meas_i\to\int \zeta\dist\meas=1$, 
from the arbitrariness of $\eps$ we obtain that $f_i(y_i)\to f(y)$. A similar argument, for arbitrary subsequences,
gives (1).

In order to prove the implication from (1) to (3) we prove the implication from (1) to (2). Assuming with no loss of
generality that $f_i$ and $f$ are nonnegative, for any $\zeta\in C_\bs(X)$ nonnegative, (1) and the compactness of the support of
$\zeta$ give that for any $\eps>0$ and any $s>0$ the set $X_i\cap \{f_i\zeta>s\}$ 
is contained in the $\epsilon$-neighbourhood of $\{f\zeta>s\}$ for $i$ large enough, so that
$$
\limsup_{i\to\infty}\meas_i(\{f_i\zeta>s\})\leq\meas(\{f\zeta\geq s\}).
$$
Analogously, any open set $A\Subset\{f\zeta>s\}$ is contained for $i$ large enough in 
the set $\{f_i\zeta>s\}\cup (X\setminus X_i)$, so that
$$
\liminf_{i\to\infty}\meas_i(\{f_i\zeta>s\})\geq\meas(\{f\zeta> s\}).
$$
Combining these two informations, Cavalieri's formula and the dominated convergence theorem provide 
$\int_X f_i\zeta\dist\meas_i\to\int_X f\zeta\dist\meas$ and then, since $\zeta$ is arbitrary, (2).

Now we can prove the implication from (1) to (3). Thanks to the equiboundedness assumption, the 
sequence $g_i:=f_i^2$ is locally equi-continuous as well and $g_i$ pointwise converge to $g:=|f|^2$ in the sense of (1), 
applying the implication from (1) to (2) for $g_i$ gives
$$
\lim_{i \to \infty}\int_{X_i}\zeta^2 f_i^2\dist \meas_i=\int_X\zeta^2 f^2\dist \meas\qquad\forall \zeta\in C_\bs(X),
$$
which yields (3).
\end{proof}

Let us recall the regularity of the heat kernel $p_{(X, \dist, \meas)}:=p(x, y, t)$ of a $\RCD^*(K, N)$-space $(X, \dist, \meas)$ we need, 
where $N \in [1, \infty)$ and $K\leq 0$. The general theory of Dirichlet forms \cite{Sturm96}, together 
with the doubling and Poincar\'e properties ensure that we can find a locally H\"older continuous representative of $p$
in $X\times X\times (0,+\infty)$, which satisfies Gaussian bounds. 
See \cite[Theorem 4]{Sturm94}, \cite[Proposition 2.3]{Sturm95}, \cite[Sections 3 and 4]{Sturm96}.

On $\RCD^*(K, N)$-spaces, finer properties of the heat kernel are known, as follows.
It was proven in \cite{JiangLiZhang} that for any $\epsilon \in (0, 1)$ there exist $C_i:=C_i(\epsilon, K, N)>1$ ($i=1,\,2)$ 
(depending only on $\epsilon,\, K,\, N$) such that 
\begin{equation}\label{eq:gaussian}
\frac{1}{C_1 \meas (B_{t^{1/2}}(x))}\exp \left(-\frac{\dist^2 (x, y)}{4(1-\epsilon)t}-C_2t \right) \le p(x, y, t) \le \frac{C_1}{\meas (B_{t^{1/2}}(x))}\exp \left( -\frac{\dist^2 (x, y)}{4(1+\epsilon)t}+C_2t \right)
\end{equation}
for all $x,\, y \in\supp\meas$ and any $t>0$. 
This, combined with the Li-Yau inequality \cite{GarofaloMondino, Jiang15} gives a gradient estimate:
\begin{equation}\label{eq:equi lip}
|\nabla_x p(x, y, t)|\le \frac{C_3}{t^{1/2}\meas (B_{t^{1/2}}(x))}\exp \left(-\frac{\dist^2(x, y)}{(4+\epsilon)t}+C_4t\right)
\end{equation}
for any $t>0$, $y\in\supp\meas$ and $\meas$-a.e. $x\in X$, 
where $C_i:=C_i(\epsilon, K, N)>1$ ($i=3,\,4$). In particular one obtains a quantitative local Lipschitz bound on $p$, 
i.e., for any $z \in X$, any $R>0$ and any $0<t_0\le t_1<\infty$ there exists $C:=C(K, N, R, t_0, t_1)>0$ such that
\begin{equation}\label{eq:lip}
|p(x, y, t)-p(\hat{x}, \hat{y}, t)| \le \frac{C}{\meas (B_{t_0^{1/2}}(z))} \dist ((x, y), (\hat{x}, \hat{y}))
\end{equation}
for all $x, \,y,\,\hat{x},\,\hat{y} \in B_R(z)\cap\supp\meas$ and any $t \in [t_0, t_1]$.
See \cite[Theorem~1.2, Corollary~1.2]{JiangLiZhang} (see also \cite{GarofaloMondino, MondinoNaber}).

The following is a generalization/refinement of Ding's result \cite[Theorems 2.6, 5.54 and 5.58]{Ding} from 
the Ricci limit setting to our setting, via a different approach.

\begin{theorem}[Pointwise convergence of heat kernels]\label{thm:heat kernel}
The heat kernels $p_i$ of $(X_i, \dist_i, \meas_i)$ satisfy
$$
\lim_{i \to \infty}p_i(x_i, y_i, t_i)=p(x, y, t)
$$
whenever $(x_i, y_i,t_i)\in X_i\times X_i\times (0,+\infty)\to (x,y,t)\in\supp\meas\times\supp\meas\times (0,+\infty)$.
\end{theorem}
\begin{proof}
By rescaling $\dist \to (t/t_i)^{1/2}\dist$, without any loss of generality we can assume that $t_i \equiv t$.
Let $f\in C_\bs(X)$ and recall that, viewing $f$ as an element of $L^2\cap L^\infty(X_i,\meas_i)$, 
$h^i_tf$ $L^2$-strongly converge to $h_tf$ \cite[Theorem~6.11]{GigliMondinoSavare13}.  
By the Bakry-Emery estimate (see for instance \cite[Theorem~6.5]{AmbrosioGigliSavare14}, here
$I_0(t)=t$ and $I_S(t):=(e^{St}-1)/S$ for $S\neq 0$) 
$$
\sqrt{2I_{2K}(t)}{\rm Lip}(h_t f,\supp\meas)\leq\|f\|_{L^\infty(X,\meas)},
$$
valid in all $\RCD(K,\infty)$ spaces, we see that $h^i_tf$ are equi-Lipschitz on $X_i$.
Thus, applying Proposition~\ref{prop:equivalence} yields $h^i_tf(y_i)\to h_tf(y)$.

On the other hand, the Gaussian estimate \eqref{eq:gaussian} shows that $\sup_i\|p_i(\cdot, y_i, t)\|_{L^{\infty}}<\infty$.
By definition, since 
$$
h_tf(y_i)=\int_{X_i}p_i(z, y_i, t)f(z)\dist \meas_i(z),\qquad h_tf(y)=\int_Xp(z, y, t)f(z)\dist \meas (z),
$$
we see that $p_i(\cdot, y_i, t)$ $L^2_{\mathrm{loc}}$-weakly converge to $p(\cdot, y, t)$.
Moreover, since thanks to \eqref{eq:lip} the functions $p_i(\cdot, y_i, t)$ are locally equi-Lipschitz continuous, 
choosing any continuous extension of $p(\cdot,y,t)$ to the whole of $X$ and
applying Proposition~\ref{prop:equivalence} once more to $p_i(\cdot, y_i,t)$ we obtain $p_i(x_i, y_i,t)$ converge to
$p(x, y, t)$, which completes the proof.
\end{proof}

\begin{definition}[$k$-dimensional regular sets $\mathcal{R}_k$ and maximal dimension $\dimnew(X)$]\label{def:regular}
Recall that the $k$-dimensional regular set $\mathcal{R}_k$ of a $\RCD^*(K, N)$-space $(X, \dist, \meas)$ in the sense of Mondino-Naber \cite{MondinoNaber} is, by definition, the set of points $x \in\supp\meas$ such that 
$$
(X, r^{-1}\dist, \meas_r^x, x) \stackrel{mGH}{\to} (\mathbb{R}^k, \dist_{\mathbb{R}^k}, c_k\mathcal{H}^k, 0_k)
$$
as $r \to 0^+$,
where $c_k$ is the normalization constant such that $\int_{B_1(0_k)}(1-|x|) \dist (c_k\mathcal{H}^k)=1$, and
$$
\meas_r^x:=\left( \int_{B_r(x)}\left(1-\frac{\dist (x, \cdot)}{r}\right) \dist \meas \right)^{-1}\meas.
$$
We denote by $\dimnew(X)$ the largest integer $k$ such that $\mathcal{R}_k$ has positive $\meas$-measure. 
\end{definition}

By the Bishop-Gromov inequality, it is easily seen that $\mathcal{R}_k=\emptyset$ if $k>[N]$. It is conjectured that
$\RCD^*(K,N)$ spaces cannot be made of pieces of different dimensions, i.e. there exists only one
integer $k$ such that $\meas(\mathcal{R}_k)>0$. This property is known to be true for Ricci limit spaces,
see \cite{ColdingNaber}.

\begin{remark}\label{rem:regular}
By the $L^2_{\mathrm{loc}}$-strong convergence of $\dist (y_i, \cdot)$ to $\dist (y, \cdot )$ for any mGH-convergent sequence $(Y_i, \dist_i, \nu_i, y_i) \stackrel{GH}{\to} (Y, \dist, \nu, y)$ of $\RCD^* (K, N)$-spaces, it is easy to check that $x \in X$ is a $k$-dimensional regular point if and only if 
$$
\left(X, r^{-1}\dist, \frac{\meas}{\meas (B_r(x))}, x\right) \stackrel{mGH}{\to} \left(\mathbb{R}^k, \dist_{\mathbb{R}^k}, \frac{\mathcal{H}^k}{\omega_k}, 0_k\right),
$$
where recall that $\omega_k$ denotes the volume of a unit ball in the $k$-dimensional Euclidean space. \end{remark}

\begin{corollary}[Short time diagonal behavior of heat kernel on the regular set]\label{cor:short time}
Let $(X, \dist, \meas)$ be a $\RCD^*(K, N)$-space with $K\in\setR$ and $N \in (1, +\infty)$. Then
\begin{equation}\label{eq:shorttime}
\lim_{t \to 0^+}\meas (B_{t^{1/2}}(x)) p(x, x, t)=\frac{\omega_k}{(4\pi)^{k/2}}
\end{equation}
for any $k$-dimensional regular point $x$ of $(X, \dist, \meas)$.
\end{corollary}
\begin{proof}
Let us recall that for any $r>0$ and any $C>0$ the heat kernel $\hat{p}(x, y, t)$ of the rescaled 
$\RCD^* (r^2K, N)$-space $(X, r^{-1}\dist, C\meas)$ is given by $\hat{p}(x, y, t)=C^{-1}p(x, y, r^2t)$.
Applying this for $r:=t^{1/2}$, $C:=\frac{1}{\meas (B_t(x))}$ with Theorem~\ref{thm:heat kernel} and Remark~\ref{rem:regular} shows
$$
\lim_{t \to 0^+}\meas (B_{t^{1/2}}(x)) p(x, x, t)=\lim_{t \to 0^+}p^t(x, x, 1)=p_{\mathbb{R}^k}(0_k, 0_k, 1)=\frac{\omega_k}{(4\pi)^{k/2}},
$$
where $p^t$, $p_{\mathbb{R}^k}$ denote the heat kernels of $\bigl(X, t^{-1/2}\dist, \frac{\meas }{\meas (B_{t^{1/2}}(x))}\bigr)$, $\left(\mathbb{R}^k, \dist_{\mathbb{R}^k}, \frac{\mathcal{H}^k}{\omega_k}\right)$, respectively. 
\end{proof}

In the proof of Weyl's law, in the next section, the following finer properties of $\mathcal{R}_k$ will be needed.

\begin{theorem}\label{thm:propRk} Let $(X,\dist,\meas)$ be a $\RCD^*(K,N)$ space with $K\in\setR$ and $N \in (1, +\infty)$.
For all $k$ the set $\mathcal{R}_k$ is $(\meas,k)$-rectifiable and 
$$
\meas\bigl(X\setminus\bigcup_{k=1}^{[N]}\mathcal{R}_k\bigr)=0.
$$
In addition, $\meas\res\mathcal{R}_k\ll\mathcal{H}^k$.
\end{theorem}
\begin{proof} See \cite{MondinoNaber} for the proof of the first two statements (more precisely, it has been
proved the stronger property that $\meas$-almost all of $\mathcal{R}_k$ can be covered by bi-Lipschitz charts with biLipschitz
constant arbitrarily close to 1, defined in subsets of the $k$-dimensional Euclidean space). See \cite{KellMondino}, \cite{DePhilippisMarcheseRindler} and \cite{GigliPasqualetto} for the proof of  the absolute continuity statement.
\end{proof}

\section{Weyl's law}

In a metric measure space $(X,\dist,\meas)$ the sequence of eigenvalues can be defined appealing to Courant's min-max procedure:
\begin{equation}\label{eq:Courant}
\lambda_i:=\min\left\{\max_{f\in S,\,\|f\|_{L^2}=1}\Ch(f):\ S\subset H^{1,2}(X,\dist,\meas),\,\,\,
{\rm dim}(S)=i\right\}\qquad i\geq 1.
\end{equation}
We then define
$$
N_{(X,\dist,\meas)}(\lambda):=\#\{i\geq 1:\ \lambda_i\leq\lambda\}
$$
as the ``inverse'' function of $i\mapsto\lambda_i$.
Notice that the formula makes sense even though $\Ch$ is not quadratic, and that the formula shows that the growth rate
of $N_{(X,\dist,\meas)}$ does not change if we replace the distance $\dist$ by a biLipschitz equivalent distance, or 
perturb the measure $\meas$ by a factor uniformly bounded away from $0$ and $+\infty$. Notice also that if $(X,\dist)$
is doubling we can always find a Dirichlet form $\mathcal E$ with $C^{-1}\mathcal E\leq\Ch\leq C\mathcal E$, with
$C$  depending only on the metric doubling constant, see \cite{AmbrosioColomboDiMarino}
(a result previously proved in \cite{Cheeger} for metric measure spaces whose measure is doubling and satisfies
Poincar\'e inequality). Thus, the replacement of $\Ch$ with $\mathcal E$ makes the standard tools of Linear Algebra
applicable.

Let us come now to Weyl's law on $\RCD$-spaces. 
It is not known yet to what extent the restriction of the measure of a $\RCD$ space (or even of a Ricci limit space) to a regular set
is quantitatively comparable to the Hausdorff measure of the corresponding dimension. As the behavior of the Hausdorff measure on the regular sets turns out to be related to the asymptotic behavior of the eigenvalues of the Laplacian, this lack of knowledge seems to be a significant difficulty to establish Weyl's law in the $\RCD$ context in full generality. However, we can bypass this difficulty for a significant class of spaces (including the class of compact Alexandrov spaces, see Corollary~\ref{cor:weyl alex}), by noticing that these spaces satisfy a suitable criterion which implies 
Weyl's law (we provide the implication in Theorem~\ref{thm:weyl}). Let us point out that all known examples of compact $\RCD$-spaces satisfy 
this criterion.

Let us start this section by giving the following, which is to some extent
a generalization of \cite[Theorem 4.6]{CheegerColding3} to the $\RCD$-setting:

\begin{theorem}[Weak Ahlfors regularity]\label{thm:RN}
Let $(X, \dist, \meas)$ be a $\RCD^* (K, N)$-space with $K \in \mathbb{R}$ and $N \in (1,+\infty)$ and set
\begin{equation}\label{eq:defRkstar}
{\mathcal R}_k^*:=\left\{x\in\mathcal{R}_k:\
\exists\lim_{r\to 0^+}\frac{\meas(B_r(x))}{\omega_kr^k}\in (0,+\infty)\right\}.
\end{equation}
Then $\meas(\mathcal{R}_k\setminus\mathcal{R}_k^*)=0$, $\mathcal{H}^N(\mathcal{R}_N\setminus\mathcal{R}_N^*)=0$ 
if $N$ is an integer, $\meas\res\mathcal{R}_k^*$ and 
$\mathcal{H}^k\res\mathcal{R}_k^*$ are mutually absolutely continuous and
\begin{equation}\label{eq:gooddensity}
\lim_{r\to 0^+}\frac{\meas(B_r(x))}{\omega_kr^k}=\frac{\dist \meas\res\mathcal{R}_k^*}{\dist \mathcal{H}^k\res\mathcal{R}_k^*}(x)
\qquad\text{for $\meas$-a.e. $x\in\mathcal{R}_k^*$.}
\end{equation}
Finally, if $k_0=\dimnew(X)$, one has
\begin{equation}\label{eq:goodlimit}
\lim_{r\to 0^+} \frac{\omega_{k_0}r^{k_0}}{\meas(B_{r}(x))}=\chi_{\mathcal{R}^*_{k_0}}(x)
\frac{\dist\mathcal{H}^k\res\mathcal{R}_{k_0}^*}{\dist\meas\res\mathcal{R}_{k_0}^*}(x)
\qquad\text{for $\meas$-a.e. $x\in X$.}
\end{equation}
\end{theorem}
\begin{proof} Let $S_k$ be a countably $k$-rectifiable subset of $\mathcal{R}_k$ with $\meas(\mathcal{R}_k\setminus S_k)=0$.
From \eqref{eq:density3} we obtain that the set $\mathcal{R}_k^*\setminus S_k$ is $\mathcal{H}^k$-negligible, hence
$\mathcal{R}_k^*$ is $(\mathcal{H}^k,k)$-rectifiable.
We denote $\meas_k=\meas\res\mathcal{R}_k$ and recall that, thanks to Theorem~\ref{thm:propRk}, 
$\meas_k\ll\mathcal{H}^k$ and $\meas=\sum_k\meas_k$. 
We denote by
$f:X\to [0,+\infty)$ a Borel function such that $\meas_k=f\mathcal{H}^k\res \mathcal{R}_k^*$ (whose existence
is ensured by the Radon-Nikodym theorem, being $\mathcal{R}_k^*$ $\sigma$-finite w.r.t. $\mathcal{H}^k$) 
and recall that \eqref{eq:densityrett} gives
\begin{equation}\label{eq:densityrett1}
\exists\lim_{r\to 0}\frac{\meas_k(B_r(x))}{\omega_k r^k}=f(x)
\qquad\text{for $\mathcal{H}^k$-a.e. $x\in \mathcal{R}_k^*$.}
\end{equation}
Now, in \eqref{eq:densityrett1} we can replace $\meas_k$ by $\meas$ for $\mathcal{H}^k$-a.e. $x\in \mathcal{R}_k^*$; this is
a direct consequence of \eqref{eq:density1} with $\mu=\meas-\meas_k$ and $S=\mathcal{R}_k^*$. 

Calling then $N_k$ the $\mathcal{H}^k$-negligible (and then $\meas_k$-negligible) subset of $\mathcal{R}_k^*$
where the equality 
$$
\lim_{r\to 0}\frac{\meas(B_r(x))}{\omega_k r^k}=f(x)
$$
fails, we obtain existence and finiteness of the limit on $\mathcal{R}_k^*\setminus N_k$;
since $f$ is a density, it is also obvious that the limit is positive $\meas_k$-a.e., and that 
$\mathcal{H}^k\res \mathcal{R}_k^*\cap\{f>0\}$ is absolutely continuous w.r.t. $\meas_k$. 

This proves that $\meas(\mathcal{R}_k\setminus\mathcal{R}_k^*)=0$ and that $\meas\res\mathcal{R}_k^*$ and 
$\mathcal{H}^k\res\mathcal{R}_k^*$ are mutually absolutely continuous. In the special case $k=N$ a suitable
density lower bound ($\meas(B_r(x))/r^N\geq\meas(X)/({\rm diam}(X))^N$ in the case $K\geq 0$, a more complex lower bound involving the
comparison spaces also holds when $K\leq 0$)
coming from the Bishop-Gromov inequality gives that $\mathcal{H}^N\ll\meas$, hence $\mathcal{R}_N\setminus\mathcal{R}_N^*$
is also $\mathcal{H}^N$-negligible. The last statement \eqref{eq:goodlimit}
follows by the fact that $r^{k_0}=o(\meas(B_r(x)))$ for $\meas$-a.e. $x\in\mathcal{R}_k$, $k<k_0$, since
this property holds on the sets $\mathcal{R}_k^*$. 
\end{proof}

Recall that, as direct consequence of standard arguments from spectral theory and elliptic regularity theory, 
for a compact $\RCD^*(K, N)$ space $(X, \dist, \meas)$ the heat kernel $p$ can be expressed by eigenfunctions:
\begin{equation}\label{eq:eigen}
p(x, y, t)=\sum_ie^{-\lambda_it}\phi_i(x)\phi_i(y)
\end{equation}
for any $x, \,y \in\supp\meas$ and any $t>0$, where $\lambda_i$ is the $i$-th eigenvalue of the Laplacian (counting with
multiplicities) and $\phi_i$ is a corresponding eigenfunction, with $\|\phi_i\|_{L^2}=1$. 
More precisely, in \eqref{eq:eigen} one choose the H\"older continuous representative of $\phi_i$, whose
H\"older norm grows linearly w.r.t. $\lambda$, so that the series in \eqref{eq:eigen} is locally
H\"older continuous in $X\times X\times (0,+\infty)$.

We are now in a position to introduce our first criterion. We always have $\mathcal{H}^k(\mathcal{R}_k^*)>0$
and, if an assumption slightly stronger than the finiteness of $k$-dimensional Hausdorff measure holds, we obtain Weyl's law
in the weak asymptotic form. For simplicity we use the notation; $f(\lambda) \sim g(\lambda)$ for the
existence of $C>1$ satisfying $C^{-1}f(\lambda) \le g(\lambda) \le Cf(\lambda)$ for sufficiently large $\lambda$.

\begin{theorem}\label{thm:weakweyl}
Let $(X, \dist, \meas)$ be a compact $\RCD^* (K, N)$-space with $K\in\setR$ and $N \in (1, +\infty)$, let $k=\dimnew(X)$ and let
$\mathcal{R}_k^*$ be as in \eqref{eq:defRkstar} of Theorem~\ref{thm:RN}.
Then we have 
\begin{equation}\label{eq: lower bound}
\liminf_{t \to 0^+}\left(t^{k/2} \sum_ie^{-\lambda_it} \right) \ge \frac{1}{(4\pi)^{k/2}}
\mathcal{H}^k({\mathcal R}_k^*)>0.
\end{equation}
%so that Remark~\ref{rem:tauberian} gives $\liminf_{\lambda \to +\infty} \lambda^{-k/2}N_{(X, \dist, \meas)}(\lambda)>0$.\footnote{SH: It should be checked.}
In particular, if $N_{(X, \dist, \meas)}(\lambda) \sim \lambda^i$ as $\lambda \to +\infty$ for some $i$, 
then Remark~\ref{rem:tauberian} gives $i \ge k/2$. In addition
\begin{equation}\label{eq:equivalencedic}
\limsup_{s\to 0^+}\int_X\frac{s^k}{\meas (B_s(x))}\dist \meas(x) <+\infty
\quad\Longleftrightarrow\quad
N_{(X, \dist, \meas)}(\lambda) \sim \lambda^{k/2}\,\,\,(\lambda\to +\infty).
\end{equation}
\end{theorem}

\begin{proof} In order to prove \eqref{eq: lower bound} we first notice that the combination of
\eqref{eq:shorttime} and \eqref{eq:goodlimit} gives
$$
\lim_{t\to 0^+}t^{k/2}p(x,x,t)=\frac{1}{(4\pi)^{k/2}}\chi_{\mathcal{R}^*_k}(x)
\frac{\dist\mathcal{H}^k\res\mathcal{R}_k^*}{\dist\meas\res\mathcal{R}_k^*}(x)
\quad\text{for $\meas$-a.e. $x\in X$.}
$$
Using the identity $t^{k/2} \sum_ie^{-\lambda_it}=\int_X t^{k/2} p(x,x,t)\dist\meas(x)$ and Fatou's lemma
we obtain 
$$
\liminf_{t \to 0}\left(t^{k/2} \sum_ie^{-\lambda_it} \right) \ge \frac{1}{(4\pi)^{k/2}}\int_{\mathcal{R}_k^*}
\frac{\dist\mathcal{H}^k\res\mathcal{R}_k^*}{\dist\meas\res\mathcal{R}_k^*}\dist\meas=
\frac{1}{(4\pi)^{k/2}}\mathcal{H}^k({\mathcal R}_k^*).
$$
The heat kernel estimate \eqref{eq:gaussian} shows
\begin{equation}\label{eq:dominated}
C^{-1}\frac{t^{k/2}}{\meas (B_{t^{1/2}}(x))}\leq
t^{k/2}p(x,x,t)\le C\frac{t^{k/2}}{\meas (B_{t^{1/2}}(x))}
\end{equation}
for some $C>1$, which is independent of $t$ and $x$. Thus the upper bound on $p$ gives
$$
\limsup_{t \to 0^+}t^{k/2}\int_Xp(x, x, t)\dist\meas(x) \leq
C\limsup_{s\to 0^+}\int_X\frac{s^k}{\meas (B_s(x))}\dist \meas(x) <+\infty.
$$
We can now invoke Remark~\ref{rem:tauberian} to obtain the implication $\Rightarrow$ in \eqref{eq:equivalencedic}.
The proof of the converse implication is similar and uses the lower bound in \eqref{eq:dominated}.
\end{proof}

Under the stronger assumption \eqref{eq:limit} (notice that both the finiteness of the limit and the 
equality of the integrals are part of the assumption) we can recover Weyl's law in the stronger form. 

\begin{theorem}\label{thm:weyl}
Let $(X, \dist, \meas)$ be a compact $\RCD^* (K, N)$-space with $K \in \mathbb{R}$ and $N \in (1, +\infty)$, and let $k=\dimnew(X)$.
Then 
\begin{equation}\label{eq:limit}
\lim_{s \to 0^+}\int_X\frac{s^{k}}{\meas (B_{s}(x))}\dist \meas(x) = 
\int_X\lim_{s \to 0^+}\frac{s^{k}}{\meas (B_{s}(x))}\dist \meas(x) <+\infty
\end{equation}
if and only if 
\begin{equation}\label{eq:weyls law}
\lim_{\lambda\to +\infty}\frac{N_{(X, \dist, \meas)}(\lambda)}{\lambda^{k/2}}
=\frac{\omega_k}{(2\pi)^k}\mathcal{H}^k(\mathcal{R}_k^*) <+\infty.
\end{equation}
\end{theorem}
\begin{proof}
We first assume that \eqref{eq:limit} holds. Taking \eqref{eq:goodlimit} and \eqref{eq:dominated}
into account, we can apply Lemma~\ref{lem:dct} with $f_t(x)=t^{k/2}p(x,x,t)$ and $g_t(x)=Ct^{k/2}/\meas (B_{t^{1/2}}(x))$ to get
\begin{align*}
\lim_{t \to 0^+}t^{k/2}\int_Xp(x, x, t)\dist \meas(x) &=\int_X \lim_{t \to 0^+}t^{k/2}p(x, x, t)\dist \meas(x) \\
&=\int_{\mathcal{R}_k^*}\frac{1}{(4\pi)^{k/2}}\frac{\dist \mathcal{H}^k\res\mathcal{R}_k^*}
{\dist\meas\res\mathcal{R}_k^*}\dist \meas \\
&=\frac{1}{(4\pi)^{k/2}}\mathcal{H}^k(\mathcal{R}_k^*)
\end{align*}
which shows \eqref{eq:weyls law} by Karamata's Tauberian theorem.

Next we assume that \eqref{eq:weyls law} holds. Then by \eqref{eq:goodlimit} and 
Karamata's Tauberian theorem again, \eqref{eq:weyls law} is equivalent to
\begin{equation}\label{eq:limit1}
\lim_{t \to 0^+}t^{k/2}\int_Xp(x, x, t)\dist \meas(x) =\int_X\lim_{t \to 0^+}t^{k/2}p(x, x, t)\dist \meas(x)<+\infty.
\end{equation}
Let $f_t(x):=t^{k/2}/\meas (B_{t^{1/2}}(x))$.
Then the heat kernel estimate \eqref{eq:dominated} shows that we can apply Lemma~\ref{lem:dct}
with $g_t(x) =C t^{k/2}p(x, x, t)$ to get \eqref{eq:limit}. 
\end{proof}

By the stability of $\RCD$-conditions with respect to mGH-convergence and \cite[Theorem 5.1]{CheegerColding1}, noncollapsed Ricci limit spaces give typical examples of $\RCD^*(K, N)$-spaces $(X, \dist, \meas)$ with $\dimnew{X}=N$. For such metric measure spaces Weyl's law was proven in \cite{Ding} by Ding. Thus the following corollary also recovers his result.

\begin{corollary}\label{cor:weyls law}
Let $(X, \dist, \meas)$ be a compact $\RCD^*(K, N)$-space with $K \in \mathbb{R}$ and $N \in (1, +\infty)$, and let $k=\dimnew{X}$. Assume that either $k=N$, or that for any integer $i$ such that $\meas (\mathcal{R}_i)>0$ there exists $g_i \in L^1(\mathcal{R}^*_i, \mathcal{H}^i)$ 
such that 
$$
g_i(x, t):=\frac{t^k}{\meas (B_t(x))}\frac{\dist \meas \res {\mathcal{R}^*_i}}{\dist \mathcal{H}^i\res {\mathcal{R}^*_i}}(x)
\leq g_i(x)\qquad\forall t\in (0,1)
$$
for $\mathcal{H}^i$-a.e. $x\in \mathcal{R}^*_i$. Then \eqref{eq:weyls law} holds.
\end{corollary}
\begin{proof} If the functions $g_i$ exist, the proof
follows by the dominated convergence theorem in conjunction with Theorem~\ref{thm:weyl}.
When $k=N$ the existence of the functions $g_i$ follows directly from the Bishop-Gromov inequality, since
$\meas(B_r(x))/r^k$ is bounded from below by a positive constant. 
\end{proof}

\begin{example}\label{rem:example1}
Let us apply Theorem~\ref{thm:weyl} to the following $\RCD^*(N-1, N)$-space:
$$
(X, \dist, \meas):=\left([0, \pi], \dist_{[0, \pi]},\sin^{N-1}t\dist t\right)
$$
for $N \in (1, \infty)$ (note that this is a Ricci limit space if $N$ is an integer, see for instance \cite{AmbrosioHonda}).
Then we can apply Theorem~\ref{thm:weyl} with $k=1$ and $\mathcal{R}_1^*=\mathcal{R}_1=(0, \pi)$, because of 
$\sup_{t<1}\|g_1(\cdot, t)\|_{L^{\infty}}<\infty$, where $g_1$ is as in Corollary \ref{cor:weyls law}.
Thus we have Weyl's law:
$$
\lim_{\lambda \to +\infty}\frac{N_{(X, \dist, \meas)}(\lambda )}{\lambda^{1/2}}=\frac{\omega_1}{2\pi}\mathcal{H}^1((0, \pi ))=1.
$$
\end{example}

\begin{example}[Iterated suspensions]\label{rem:example2}
Let us apply now Theorem~\ref{thm:weyl} to iterated suspensions of $(X, \dist, \meas)$ as in Example~\ref{rem:example1}:
$$
\begin{cases}
(X_1, \dist_1, \meas_1):= ([0, \pi], \dist_{[0, \pi]},\sin^{N-1}t \dist t),\\
 (X_{n+1}, \dist_{n+1}, \meas_{n+1}):=([0, \pi], \dist_{[0, \pi]},\sin t \dist t)\times^1 (X_n, \dist_n, \meas_n).
\end{cases}
$$

Recall that the spherical suspension $([0, \pi], \dist_{[0, \pi]},\sin t \dist t)\times^1 (X, \dist, \meas)$ of a metric measure space $(X,\dist,\meas)$ is the quotient of the product $[0,\pi] \times X$ by the identification of every point of $\{0\}\times X$ and $\{\pi\}\times X$ into two distinct points, equipped with the product measure $\dist \mu := \sin t\dist t\times \meas$ and with the distance $\dist_{\mathrm{susp}}$ defined by 
$$
\cos \dist_{\mathrm{susp}}\bigl((t,x),(s,y)\bigr) = \cos t \cos s + \sin t \sin s \cos (\min \{\dist (x,y), \pi\}).
$$

Note that $(X_n, \dist_n, \meas_n)$ is a $\RCD^*(N+n-2, N+n-1)$-space (see \cite{Ketterer15a}) and that $(X_n, \dist_n)$ are isometric to a hemisphere of the $n$-dimensional unit sphere $\mathbb{S}^n(1)$ as metric spaces.

Then we can apply Theorem~\ref{thm:weyl}  because an elementary calculation similar to the one of
Example~\ref{rem:example1} shows that
%as the density function $f$ of $\meas_n$ with respect to $\mathcal{H}^n$ (i.e. $\dist \meas_n=f \dist \mathcal{H}^n$) is continuous on $X_n$ and is positive on the regular set, we have 
$\sup_{t<1}\|g_n(\cdot , t)\|_{L^\infty}<\infty$.
Thus Weyl's law follows:
$$
\lim_{\lambda \to +\infty}\frac{N_{(X_n \dist_n, \meas_n)}(\lambda)}{\lambda^{n/2}}=\frac{\omega_n}{(2\pi)^n}\mathcal{H}^n(X_n)=\frac{\omega_n}{(2\pi)^n}\frac{\mathcal{H}^n(\mathbb{S}^n(1))}{2}.
$$
\end{example}

\begin{example}[Gaussian spaces]\label{rem:example3}
For noncompact $\RCD(K, \infty)$-spaces the behavior of the spectrum is different, and requires a more delicate analysis. 
For instance (see \cite[(2.2)]{Milman}) the $n$-dimensional Gaussian space 
$(X, \dist, \meas):=(\mathbb{R}^n, \dist_{\mathbb{R}^n}, \gamma_n)$
satisfies
$$
\lim_{\lambda \to +\infty}\frac{N_{(X, \dist, \meas)}(\lambda )}{\lambda^n}=\frac{1}{\Gamma (n+1)}.
$$
\end{example}

\begin{corollary}[Weyl's law on compact Ahlfors regular $\RCD^*(K, N)$-spaces - especially Alexandrov spaces]\label{cor:weyl alex}
Let $(X, \dist, \meas)$ be a compact $\RCD^*(K, N)$-space with $K \in \mathbb{R}$ and $N \in (1, +\infty)$.
Assume that $(X, \dist, \meas)$ is Ahlfors $n$-regular for some $n \in \mathbb{N}$, i.e. 
there exists $C>1$ such that $$C^{-1}r^n \le \meas (B_r(x)) \le Cr^n\qquad\forall x \in X,\,\,r \in (0, 1).$$
Then we have Weyl's law:
\begin{equation}\label{eq:weyl ahlfors}
\lim_{\lambda \to +\infty}\frac{N_{(X, \dist, \meas)}(\lambda)}{\lambda^{n/2}}=\frac{\omega_n}{(2\pi)^n}\mathcal{H}^n(X).
\end{equation} 
In particular this holds if $(X, \dist, \meas)$ is an $n$-dimensional compact Alexandrov space.
%\begin{equation}\label{eq:weyl alex}
%\lim_{\lambda \to +\infty}\frac{N_{(Y, \dist, \mathcal{H}^n)}(\lambda)}{\lambda^{n/2}}=\frac{\omega_n}{(2\pi)^n}\mathcal{H}^n(Y).
%\end{equation}
\end{corollary}
\begin{proof}
Note that by the Ahlfors $n$-regularity of $(X, \dist, \meas)$, any tangent cone at $x$ also satisfies the Ahlfors $n$-regularity, which implies that $\mathcal{R}_i = \emptyset $ for any $i \neq n$.
In particular since $\mathcal{H}^n \ll \meas \ll \mathcal{H}^n$, we have 
\begin{equation}\label{eq:n-dim}
\meas (X \setminus \mathcal{R}_n)=\mathcal{H}^n(X \setminus \mathcal{R}_n)=0.
\end{equation}
Then Theorem~\ref{thm:weyl} can be applied with $g_k \equiv c$ for some $c>0$, which proves \eqref{eq:weyl ahlfors} by \eqref{eq:n-dim}.
The final statement follows from the compatibility between Alexandrov spaces and $\RCD$-spaces \cite{Petrunin, ZhangZhu}.
\end{proof}

\begin{example}\label{ex:weighted weyl}
Let us discuss the simplest case we can apply Corollary \ref{cor:weyl alex}; let $M$ be a compact $n$-dimensional manifold and let $f \in C^2(M)$.
Then, thanks to (\ref{eq:RCD smooth metric measure space}), for any $N \in (n, \infty)$ there exists 
$K\in\mathbb{R}$ such that $(M, \dist, e^{-f}\mathcal{H}^n)$ is a $\RCD^*(K, N)$-space.
Moreover since $\left(M, \dist, e^{-f}\mathcal{H}^n\right)$ is Ahlfors $n$-regular, 
Corollary~\ref{cor:weyl alex} yields Weyl's law:
$$
\lim_{\lambda \to +\infty}\frac{N_{(M, \dist, e^{-f}\mathcal{H}^n)}(\lambda)}{\lambda^{n/2}}
=\frac{\omega_n}{(2\pi)^n}\mathcal{H}^n(M).
$$
\end{example}

In order to give another application of Weyl's law on compact finite dimensional Alexandrov spaces, let us recall that two compact finite dimensional Alexandrov spaces are said to be \textit{isospectral} if the spectrums of their Laplacians coincide.
See for instance \cite{Sunada, EngelWeiandt} for constructions of isospectral manifolds and of isospectral Alexandrov spaces (see also \cite{KuwaeMachigashiraShioya} for analysis on Alexandrov spaces).
 
It is also well-known as a direct consequence of Perelman's stability theorem \cite{Perelman} (see also \cite{Kapovitch}) that for fixed $n \in \mathbb{N}$, $K \in \mathbb{R}$ and $d, \,v >0$ the isometry class of $n$-dimensional compact Alexandrov spaces $X$ of sectional curvature bounded below by $K$ with $\mathrm{diam}\,X \le d$ and $\mathcal{H}^n(X) \ge v$ has only finitely many topological types.
By using this and Weyl's law, we have
the following which is a generalization of topological finiteness results for isospectral spaces proven in \cite{BrooksPetersenPerry, Harvey, Stanhope} to Alexandrov spaces.

\begin{corollary}[Topological finiteness theorem for isospectral Alexandrov spaces]\label{cor:finiteness}
Let $\chi:=\{(X_u, \dist_u, \mathcal{H}^{n_u})\}_{u\in U}$ be a class of compact finite dimensional Alexandrov spaces with a uniform sectional curvature bound from below. Assume that there exists $C>1$ such that
\begin{equation}\label{eq:equi bounded}
\limsup_{\lambda\to +\infty}\frac{N_{(X_u, \dist_u, \mathcal{H}^{n_u})}(\lambda)}
{N_{(X_v, \dist_v, \mathcal{H}^{n_v})}(\lambda)}\le C
\end{equation}
for all $u,\,v \in U$.
Then $\chi$  has only finitely many topological types.

In particular, any class of isospectral compact finite dimensional Alexandrov spaces with a uniform
sectional curvature bound from below has only finitely many members up
to homeomorphism.
\end{corollary}
\begin{proof}
By an argument similar to the proof of \cite[Corollary 1.2]{BrooksPetersenPerry} (or \cite[Proposition 7.4]{Stanhope}) with
\cite[Corollary 1]{VonRenesse} 
%\cite[Corollary 4.4]{KuwaeShioya}, 
there exists $d>0$ such that $\mathrm{diam}\,X_{\lambda}\le d$. 
Since Weyl's law (\ref{eq:weyl ahlfors}) with (\ref{eq:equi bounded}) implies that there exist $n \in \mathbb{N}$ and $v >0$ such that $\mathrm{dim}\,X_{\lambda}\equiv n$ and $\mathcal{H}^n(X_{\lambda}) \ge v$ for any $\lambda \in \Lambda$, the topological finiteness result stated above completes the proof.
\end{proof}

\section{Appendix: refinements of Karamata's theorem}

In this section we prove Theorem~\ref{thm:karamata} and its one-sided versions mentioned in Remark~\ref{rem:tauberian}.
We follow the proofs in Theorems 10.2 and 10.3 of \cite{S}, borrowing also the terminology
``Abelian'', ``Tauberian'' from there.

Throughout this section $\nu$ is a nonnegative and $\sigma$-finite Borel measure on $[0,+\infty)$.
The results will then be applied to the case when $\nu:=\sum_i\delta_{\lambda_i}$.

\begin{lemma}\label{lem:int by parts}
For all $t>0$ one has
\begin{equation}\label{eq:int}
\int_{[0,+\infty)} e^{-tx}\dist \nu(x)=\int_0^\infty t\nu([0,y])e^{-ty}\dist y.
\end{equation}
\end{lemma}
\begin{proof} By Cavalieri's formula and the change of variables $r=e^{-ty}$ we get
$$
\int_{[0,+\infty)} e^{-tx}\dist \nu(x)=\int_0^1\nu(\{x:\ e^{-tx}\geq r\})\dist r=
\int_0^\infty te^{-ty}\nu(\{x:\ e^{-tx}\geq e^{-ty}\})\dist y
$$
and we conclude, since $\{x:\ e^{-tx}\geq e^{-ty}\}=[0,y]$.
\end{proof}

We start with the Abelian case, easier when compared to the Tauberian one.

\begin{theorem}[Abelian theorem]
Assume that there exist $\gamma \in [0,+\infty)$ and $C \in [0,+\infty)$ such that 
\begin{equation}\label{1}
\lim_{a \to +\infty}\frac{\nu ([0, a])}{a^{\gamma}}=C.
\end{equation}
Then 
\begin{equation}\label{2}
\lim_{t \to 0^+}t^{\gamma}\int_{[0,+\infty)} e^{-tx}\dist \nu (x)=C \Gamma (\gamma +1).
\end{equation}  
More generally,
\begin{equation}\label{1bis}
\limsup_{a \to +\infty}\frac{\nu ([0, a])}{a^{\gamma}}\leq C<+\infty\quad\Longrightarrow\quad
\limsup_{t \to 0^+}t^{\gamma}\int_{[0,+\infty)} e^{-tx}\dist \nu (x)\leq C \Gamma (\gamma +1)
\end{equation}
and
\begin{equation}\label{1ter}
\liminf_{a \to +\infty}\frac{\nu ([0, a])}{a^{\gamma}}\geq c\quad\Longrightarrow\quad
\liminf_{t \to 0^+}t^{\gamma}\int_{[0,+\infty)} e^{-tx}\dist \nu (x)\geq c \Gamma (\gamma +1).
\end{equation}
\end{theorem}
\begin{proof}
Let $F(a):=\nu ([0, a])$ and $G(a):=(a+1)^{-\gamma}F(a)$.
Then (\ref{1}) yields 
\begin{equation}\label{3}
\lim_{a \to +\infty}G(a)=C.
\end{equation}
In particular $\sup_a G(a)<\infty$. Then Lemma~\ref{lem:int by parts} gives
\begin{equation}\label{4}
t^{\gamma}\int_{[0,+\infty)} e^{-tx}\dist \nu (x) = 
t^{\gamma +1}\int_0^\infty e^{-tx}(x+1)^{\gamma}G(x)\dist x=
\int_0^\infty e^{-y}(y+t)^{\gamma}G\left( y/t \right) \dist y.
\end{equation}
Since for any $t \in (0, 1]$
\begin{equation}\label{eq:333}
e^{-y}(y+t)^{\gamma}G(y/t) \le e^{-y}(y+1)^{\gamma} \sup_a G(a) \in L^1([0,+\infty)),
\end{equation}
applying the dominated convergence theorem to \eqref{4} as $t \downarrow 0$ shows \eqref{2} 
because $G(y/t) \to C$ as $t \downarrow 0$ by \eqref{3}.

The one-sided versions \eqref{1bis}, \eqref{1ter} follow by an analogous argument, 
using Fatou's lemma and noticing that in the $\limsup$ case the functions in \eqref{eq:333} are dominated
as $t\to 0^+$ by an integrable function.
\end{proof}

Now we deal with the Tauberian case.

\begin{theorem}[Tauberian theorem]
Assume that there exist $\gamma \in [0,+\infty)$ and $D \in [0,+\infty)$ such that 
\begin{equation}\label{5}
\lim_{t \to 0^+}t^{\gamma}\int_{[0,+\infty)} e^{-tx}\dist \nu (x)=D.
\end{equation}
Then
\begin{equation}\label{6}
\lim_{a \to +\infty}\frac{\nu ([0, a))}{a^{\gamma}}=\frac{D}{\Gamma (\gamma +1)}.
\end{equation}
\end{theorem}
\begin{proof}
If $\gamma =0$, then applying the monotone convergence theorem to (\ref{5}) shows (\ref{6}), hence
we can assume $\gamma >0$. For any $t \in (0, 1]$ let $\nu_t,\, \mu$ be Borel measures on 
$[0,+\infty)$ be respectively defined by  
\begin{equation}\label{7}
\nu_t(A):=t^{\gamma}\nu(t^{-1}A),\qquad
\mu (A):=\int_Ax^{\gamma -1}\dist x
\end{equation}
for any Borel subset $A$.
Then (\ref{6}) is equivalent to 
\begin{equation}\label{7}
\lim_{t \to 0^+}\nu_t([0, 1))=\frac{D}{\Gamma (\gamma)}\mu ([0, 1))
\end{equation}
because
\begin{equation}\label{9}
\nu_t([0, 1))=t^{\gamma}\nu ([0, t^{-1}))\quad\text{and}\quad
\mu ([0, 1))=\int_0^1x^{\gamma -1}\dist x=\frac{1}{\gamma}=\frac{\Gamma (\gamma )}{\Gamma (\gamma +1)}.
\end{equation}
In order to prove (\ref{7}), we will show
\begin{equation}\label{11}
\lim_{t \to 0^+}\int f(x)\dist \nu_t(x)=\frac{D}{\Gamma (\gamma)}\int f(x)\dist \mu(x)
\end{equation} 
for any $f \in C_c([0,+\infty))$ as follows.

Let $\hat{\nu}_t:=e^{-x}\dist \nu_t(x)$ and $\hat{\mu}:=e^{-x}\dist \mu(x)$ be the corresponding 
weighted measures on $[0,+\infty)$.
Then (\ref{5}) with Lemma~\ref{lem:int by parts} yields
\begin{equation}\label{12}
\lim_{t \to 0^+}\hat{\nu}_t([0,+\infty))=\lim_{t\to 0^+}\int e^{-x}\dist \nu_t(x) 
=\lim_{t \to 0^+}\int e^{-tx}t^{\gamma}\dist \nu (x)=\frac{D}{\Gamma (\gamma)}\hat{\mu}([0,+\infty)).
\end{equation}
In particular 
\begin{equation}\label{13}
\sup_{t<1}\hat{\nu}_t([0, +\infty))<+\infty.
\end{equation}
More strongly, (\ref{5}) yields
\begin{equation}\label{14}
\lim_{t \to 0^+}\int g(x)\dist \hat{\nu}_t(x)=\frac{D}{\Gamma (\gamma)}\int g(x)\dist \hat{\mu}(x)
\end{equation}
for any polynomial $g(x)$ in $e^{-x}$ (i.e. $g(x)=\sum_{i=1}^Na_ie^{-ix}$).
Because 
\begin{align}
\lim_{t \to 0^+}\int e^{-kx}\dist \hat{\nu}_t(x)&=\lim_{t \downarrow 0}\int e^{-(k+1)x}\dist \nu_t(x) \nonumber \\
&=\lim_{t \to 0^+}\int e^{-(k+1)tx}t^{\gamma}\dist \nu(x) \nonumber \\
&=\frac{D}{(k+1)^{\gamma}}\nonumber =\frac{D}{\Gamma (\gamma)}\int e^{-kx}\dist \hat{\mu}(x).
\end{align}

Let $C_0([0,+\infty))$ be the set of continuous functions $f$ on $[0,+\infty)$ such that $f(x) \to 0$ as $x \to +\infty$.
Then since the set of polynomials in $e^{-x}$ is dense in $C_0([0,+ \infty))$ with respect to the norm $\sup |f|$, applying the Stone-Weierstrass theorem to $(C_0([0,+\infty)), \sup |\cdot |)$ with \eqref{13} shows that \eqref{14} is satisfied for any $g \in C_0([0,+\infty))$, which implies \eqref{11}.

We are now in a position to prove (\ref{7}) by using (\ref{11}). Indeed, it is well-known that
the weak convergence implies $\nu_t(E)\to D\mu(E)/\Gamma(\gamma)$ for any compact set $E\subset [0,+\infty)$ 
with $\mu(\partial E)=0$. Choosing $E=[0,1]$ we obtain \eqref{11}.
\end{proof}

\begin{remark}\label{rem:moments}
The difficulty to obtain a one sided version out of the previous proof, as we did for the Abelian case,
can also be explained as follows: if we consider the push forward $\sigma_t$ of the measures $\hat{\nu_t}$
under the map $x\mapsto e^{-x}$, the argument above shows that all moments of all weak limit points
of $\sigma_t$ are uniquely determined. Hence, since a finite Borel measure in $[0,1]$ is uniquely determined by
its moments, uniqueness follows. If we replace the assumption \eqref{5} by a bound on the $\liminf$ or
the $\limsup$, we find only an inequality between the moments of the measures, which does not seem to imply, 
in general, the corresponding inequality for the measures.
\end{remark}

\begin{proposition}\label{prop}
Assume that for some $\gamma \in [0, +\infty)$ one has 
\begin{equation}\label{24}
\limsup_{t \to 0^+}t^{\gamma}\int_{[0,+\infty)}e^{-st}\dist\nu(s)\le C_0<+\infty.
\end{equation}
Then 
\begin{equation}\label{25}
\limsup_{\lambda \to +\infty}\frac{\nu([0,\lambda])}{\lambda^{\gamma}}\le eC_0.
\end{equation}
\end{proposition}
\begin{proof}
Note that for any $\lambda>0$ and any $t>0$
\begin{equation}\label{26}
\nu([0,\lambda])\le e^{\lambda t}\int_{[0,\lambda]}e^{-st}\dist\nu(s)
\le e^{\lambda t}\int_{[0,+\infty)}e^{-st}\dist\nu(s).
\end{equation}
By \eqref{24}, for any $\epsilon>0$ there exists $t_0>0$ such that
$\int_{[0,+\infty)}e^{-st}\dist\nu(s)\le (C_0+\epsilon)t^{-\gamma}$ for any $t<t_0$.
Thus \eqref{26} yields $\nu([0,\lambda])\le e^{\lambda t}(C_0+\epsilon)t^{-\gamma}$
for any $\lambda>0$ and any $t<t_0$.
Letting $\lambda:=t^{-1}$ and then letting $t \downarrow 0$ shows \eqref{25}.
\end{proof}

\begin{proposition}
Assume that for some $\gamma \in [0,+\infty)$ one has 
\begin{equation}\label{33}
\liminf_{t \to 0^+}t^{\gamma}\int_{[0,+\infty)}e^{-st}\dist\nu(s)>0,\qquad
\limsup_{t \to 0^+}t^{\gamma}\int_{[0,+\infty)}e^{-st}\dist\nu(s)<+\infty.
\end{equation}
%
%\begin{equation}\label{34}
%\limsup_{t \downarrow 0}t^{\gamma}\sum_ie^{-\lambda_it}\le C_1.
%\end{equation}}
Then 
\begin{equation}\label{35}
\liminf_{\lambda \to +\infty}\frac{\nu([0,\lambda])}{\lambda^{\gamma}} >0.
\end{equation}
\end{proposition}
\begin{proof} Call $C_0>0$ the $\liminf$ and $C_1<+\infty$ the $\limsup$ in \eqref{33}.
Note that for any $\lambda >0$ and any $t>0$
\begin{align*}
\int_{[0,+\infty)}e^{-st}\dist\nu(s)&=\int_{[0,\lambda]}e^{-st}\dist\nu(s) + 
\sum_{\ell=1}^{\infty}\int_{(\ell\lambda,(\ell+1)\lambda]}e^{-st}\dist\nu(s) \\
%& \le \nu([0,\lambda])+\sum_{\ell=1}^{\infty}\left( \sum_{l\lambda <\lambda_i \le (\ell+1)\lambda}e^{-\ell\lambda t}\right) \\
%&=  \nu([0,\lambda])+ \sum_{\ell=1}^{\infty}\left(e^{-\ell\lambda t}(N((\ell+1)\lambda)-N(\ell\lambda))\right)  \\
&\le \nu([0,\lambda])+ \sum_{\ell=1}^{\infty}e^{-\ell\lambda t}\nu([0,(\ell+1)\lambda]) \\
&=  \sum_{\ell=0}^{\infty}e^{-\ell\lambda t}\nu([0,(\ell+1)\lambda]).
\end{align*}
In particular, letting $\lambda:=t^{-1}$ yields
\begin{equation}\label{36}
t^{\gamma}\int_{[0,+\infty)}e^{-st}\dist\nu(s)\le t^{\gamma}\sum_{\ell=0}^{\infty}e^{-\ell}\nu([0,\frac{\ell+1}{t}]).
\end{equation}
Thus there exists $t_0>0$ such that for any $t<t_0$
\begin{equation}\label{37}
0<\frac{C_0}{2} \le t^{\gamma}\sum_{\ell=0}^{\infty}e^{-\ell}\nu([0,\frac{\ell+1}{t}]).
\end{equation}
Next let us discuss the right hand side of (\ref{37}). By \eqref{33} and Proposition~\ref{prop} there exists $\hat{\lambda}>0$ such that $\nu([0,\lambda])\le (eC_1+1)\lambda^{\gamma}$ for any $\lambda \ge \hat{\lambda}$.
Thus for any $t>0$ with $t^{-1} \ge \hat{\lambda}$ we get
$$
\nu([0,\frac{\ell+1}{t}])\le (eC_1+1) \frac{(\ell+1)^{\gamma}}{t^{\gamma}}.
$$
In particular 
\begin{equation}\label{39}
t^{\gamma}\sum_{\ell=k}^{\infty} e^{-\ell}\nu([0,\frac{\ell+1}{t}])\le 
(eC_1+1)\sum_{\ell=k}^{\infty}e^{-\ell}(\ell+1)^{\gamma}
\end{equation}
for any $k \in \mathbb{N}$ and any $t>0$ with $t^{-1}\ge \hat{\lambda}$.

For any $\delta >0$ there exists $k_0 \in \mathbb{N}$ such that 
$\sum_{\ell=k_0+1}^{\infty}e^{-\ell}(\ell+1)^{\gamma}<\delta$.
Then, combining \eqref{37} with \eqref{39} yields
\begin{equation}
0<\frac{C_0}2<t^{\gamma}\sum_{\ell=0}^{k_0}e^{-\ell} \nu([0,\frac{\ell+1}{t}])+(eC_1+1)\delta
\end{equation}
for any $t>0$ with $t<t_0$ and $t^{-1} \ge \hat{\lambda}$, which easily shows \eqref{35}
choosing $\delta>0$ so small that $(eC_1+1)\delta<C_0/2$.
\end{proof}

\end{document}